\documentclass[11pt,reqno]{amsart}
\usepackage{amssymb}
\usepackage{amsmath}
\usepackage[active]{srcltx}
\usepackage{t1enc}
\usepackage[latin2]{inputenc}
\usepackage{verbatim}
\usepackage{amsmath,amsfonts,amssymb,amsthm}
\usepackage[mathcal]{eucal}
\usepackage{enumerate}
\usepackage[centertags]{amsmath}
\usepackage{graphics}

\newtheorem{theorem}{Theorem}

\newtheorem{lemma}{Lemma}

\newtheorem{definition}{Definition}

\begin{document}
	\author{V. Tsagareishvili and G. Tutberidze}
	\title[Some problems of convergence of general Fourier series]{Some problems of convergence of general Fourier series}
	\address{ Associate Professor V. Tsagareishvili, Department of Mathematics, Faculty of Exact and Natural Sciences, Ivane Javakhishvili Tbilisi State University, Chavchavadze str. 1, Tbilisi 0128, Georgia }
	\email{cagare@ymail.com}
	\address{G.Tutberidze, The University of Georgia, School of science and technology, 77a Merab Kostava St, Tbilisi 0128, Georgia}
	\email{g.tutberidze@ug.edu.ge \and giorgi.tutberidze1991@gmail.com}
	
	\thanks{The research was supported by Shota Rustaveli National Science Foundation grant  no. FR-19-676.}
	\date{}
	\maketitle

	\begin{abstract}
		S. Banach \cite{Banach} proved that good differential properties of  function do not guarantee the a.e. convergence of the Fourier series of this function with respect to general orthonormal systems (ONS). On the other hand it is very well known that a sufficient condition for the a.e. convergence of an orthonormal series is given by the Menshov-Rademacher Theorem.
		
		The paper deals with sequence of positive numbers $(d_n)$ such that multiplying the Fourier coefficients $(C_n(f))$ of functions with bounded variation by these numbers one obtains a.e. convergent series of the form $\sum_{n=1}^{\infty}d_n C_n(f) \varphi_n (x).$ It is established that the resulting conditions are best possible.
	\end{abstract}
	
	

\section{SOME NOTATIONS AND THEOREMS}
Let $\left(\varphi_n\right)$ be an orthonormal system (ONS) on $\left[0,1\right]$ and
\begin{eqnarray}
	\label{1.1} C_n(f)=\int_0^1 f(x) \varphi_n\left(x\right)\,dx, \quad n=1,2 
\end{eqnarray}
be the Fourier coefficients of a function $f\in L_2 \left(0,1\right).$

We denote by $V\left(0,1\right)$ the class of all functions of bounded variation and write $V \left(f\right)$ for the total variation of a function $f$ on $\left[0,1\right]$.

Let $A$ be the class of all absolutely continuous functions $f$ on $[0, 1]$. This is a Banach space with the norm
$$\left\Vert f \right\Vert _A = \int_{0}^{1} \left|f^{'} \left(x\right)\right|dx+\left \Vert f\left(x\right) \right \Vert_C,$$
where $C(0, 1)$ is the class of all continuous functions $f$ on $[0, 1];$ $\left \Vert f(x) \right \Vert_C$ is the norm of $f$ on $C(0, 1)$.

\begin{definition} \label{definition1.1}
	A positive bounded sequence of numbers $\left(d_n\right)$ is called a multiplier of convergence with respect to a function class $E$ if 
	\begin{eqnarray}
		\qquad \sum_{k=1}^{\infty} d_k C_k\left(f\right) \varphi_k\left(x\right), \qquad \left(C_k\left(f\right)=\int_{0}^{1} f\left(x\right)\varphi_k\left(x\right)dx\right)	\notag
	\end{eqnarray}
	is convergence a.e. for all $ f\in E.$
\end{definition}

\begin{theorem} \label{theorem1.1} (see \cite{KashinSaakyan} ch.9.p332)(Menshov-Rademacher).
	If $\left(\varphi_n\right)$ is an ONS on $\left[0,1\right]$ and a number sequence $\left(c_n\right)$ satisfies the condition
	\begin{eqnarray}
		\sum_{n=1}^{\infty}c_n^2 \log_2^2 n < +\infty, \notag
	\end{eqnarray}
	then the series 
	\begin{eqnarray} \notag
		\sum_{n=1}^{\infty}c_n \varphi_n\left(x\right)
	\end{eqnarray}
	converges a.e. on $\left[0,1\right]$.
\end{theorem}

\begin{lemma}[see \cite{GogoladzeTsagareishvili}]\label{lemma1.1} If $f \in L_2 \left(0,1\right)$ takes only finite values on $\left[0,1\right]$ and $g \in L_2 \left(0,1\right)$  is an arbitrary function, then
	
	\begin{eqnarray} \label{1.2}
		\int_{0}^{1} f\left(x\right) g \left(x\right) dx 
		 &=& \sum_{i=1}^{n-1}\left(f\left(\frac{i}{n}\right)-f\left(\frac{i+1}{n}\right)\right)\int_{0}^{i/n}g\left(x\right)dx \\
		&+& \sum_{i=1}^{n}\int_{\left(i-1\right)/n}^{i/n}\left(f\left(x\right) -f\left(\frac{i}{n}\right)\right)g\left(x\right)dx   \notag \\
		&+&f\left(1\right)\int_{0}^{1}g\left(x\right)dx.     \notag
	\end{eqnarray}
\end{lemma}
We have that $(\log n := \log_2 n)$
\begin{eqnarray} \label{1.3} 
	\sum_{k=1}^{\infty}d_k^2 C_k^2\left(f\right) \log^2 k &=&\sum_{k=1}^{\infty}d_k^2 C_k\left(f\right) C_k\left(f\right) \log^2 k \\
	&=&\int_{0}^{1}f\left(x\right)\sum_{k=1}^{n}d_k^2 C_k\left(f\right) \log^2 k \varphi_k\left(x\right)dx \notag \\
	&=& \int_{0}^{1} f\left(x\right) P_n\left(d,c,x\right)dx,  \notag 
\end{eqnarray}
 where $c=(C_n(f)), \ \ a=(a_n)$ and
$$P_n\left(d,a,x\right)=\sum_{k=1}^{n}d_k^2 a_k \log^2 k \varphi_k\left(x\right).$$
Set
\begin{eqnarray} \label{1.4}
	G_n\left(d,a\right)=\max_{1\leq i \leq n} \left\vert \int_{0}^{i/n} P_n\left(d,a,x\right)dx\right\vert
\end{eqnarray}
and
\begin{equation} \label{1.5}
	T_n\left(d,a\right)=\left(\sum_{k=1}^{n} d_k^2 a_k^2 \log^2 k\right)^{1/2},
\end{equation}
where $\left(a_n\right)\in l_2.$

\begin{lemma}\label{lemma1.2}
Let $(d_n )$  be a positive, bounded sequence of numbers.	Then  for every $i, \left(i=1,2, \dots, n \right)$
	\begin{equation*}
		\int_{\left(i-1\right)/n}^{i/n}\left\vert P_n\left(d,a,x\right)\right \vert dx =O \left(1\right) T_n \left(d,a\right).
	\end{equation*}
\end{lemma}

\begin{proof}
	If we use the Cauchy inequality and we  mean that $$D=\sup_{k} d_k$$
	we get 
	\begin{eqnarray}
		\int_{\left(i-1\right)/n}^{i/n}\left\vert P_n\left(d,a,x\right)\right \vert dx &\leq& \frac{1}{\sqrt{n}}\left(\int_{0}^{1}P_n^2 \left(d,a,x\right)dx\right)^{1/2}	\notag \\
		&=& \frac{1}{\sqrt{n}} \left(\int_{0}^{1}\left(\sum_{k=1}^{n} d_k^2 a_k \log^2 k \varphi_k\left(x\right)\right)^2dx\right)^{1/2} \notag \\
		&=& \frac{1}{\sqrt{n}} \left(\sum_{k=1}^{n} d_k^4 a_k^2 \log^4 k\right)^{1/2} \notag \\
		&\leq& D\frac{\log n}{\sqrt{n}}\left(\sum_{k=1}^{n} d_k^2 a_k^2 \log^2 k\right)^{1/2}=O\left(1\right)T_n\left(d,a\right).	\notag 
	\end{eqnarray}
	Lemma \ref{lemma1.2} is proved.
	
\end{proof}

\section{STATEMENT OF THE MAIN PROBLEM}
General ONS were studied by a lot of authors. We mention Gogoladze and Tsagareishvili \cite{GogoladzeTsagareishvili}-\cite{GogoladzeTsagareishvili_5}, Kashin and Saakyan \cite{KashinSaakyan}, Tsagareishvili and Tutberidze \cite{tsatut1,tsatut2}. Convergence and summability of Fourier series with respect to Walsh, Vilinkin, Haar and trigonometric systems were studied by Gogoladze and Tsagareishvili \cite{GogoladzeTsagareishvili_6}, Persson, Tephnadze and Tutberidze \cite{PTT} (see also \cite{BPT}, \cite{PSTW}), Tephnadze \cite{tep1}-\cite{tep4}, Tutberidze \cite{tut1}-\cite{tut3}. Similar problems for the two-dimensional case can be found in Goginava and Gogoladze \cite{gg1}, Persson, Tephnadze and Wall \cite{PTW}.

 From Banachs Theorem \cite{Banach} it follows that if  $f \in L_2 (0,1),\ \left(f\nsim 0\right)$ then there exists an ONS such that  the  Fourier series of  this function $f$ is not convergent on $[0,1]$ with respect to this system. Thus it is clear that the Fourier coefficients of functions of bounded variation in general do not satisfy  condition of  Theorem \ref{theorem1.1}. In the present paper we have studied  the sequence  $(d_n)$ so that the Fourier coefficients  of every function from $V(0,1)$ satisfy the condition
$$\sum_{n=1}^{\infty}d_n^2 C_n^2 \left(f\right) \log^2 n <+\infty.$$
The similar results are obtained in \cite{GogoladzeTsagareishvili}-\cite{GogoladzeTsagareishvili_5}.

\section{The Main Results}
\begin{theorem}\label{theorem3.1}
	Let $\left(\varphi_n\right)$ be an ONS on $\left[0,1\right]$ and $(d_n)$ is a given sequence of numbers. If for any $(a_n)\in l_2 $ 
	\begin{eqnarray} \label{3.1}
		\qquad	G_n\left(d,a\right)=O\left(1\right)T_n\left(d,a\right),
	\end{eqnarray}
 then for every $f\in V(0,1)$

\begin{eqnarray*}
	\sum_{n=1}^{\infty}d_n^2 C_n^2 (f) \log ^2 n < \infty.
\end{eqnarray*}
\end{theorem}

\begin{proof}
	By using Lemma\ref{lemma1.1}, when $g\left(x\right) = P_n \left(d,c,x\right)$ we have ($c=\left(C_k\left(f\right)\right)$)
	\begin{eqnarray}
		\label{3.2} \qquad  &&\int_{0}^{1} f\left(x\right) P_n\left(d,c,x\right)dx \\ &&= \sum_{i=1}^{n-1} \left(f\left(\frac{i}{n}\right)-f\left(\frac{i+1}{n}\right)\right) \int_{0}^{i/n}P_n\left(d,c,x\right)dx \notag \\ 
		&&+ \sum_{i=1}^{n-1} \int_{\left(i-1\right)/n}^{i/n} \left(f\left(x\right)-f\left(\frac{i}{n}\right) \right) P_n \left(d,c,x\right)dx  \notag
		\\ 
		&&+ f(1)\int_{0}^{1}P_n\left(d,c,x\right)dx. \notag
	\end{eqnarray}
	Next (see \eqref{1.3})
	\begin{eqnarray}
		\label{3.3} \qquad \sum_{k=1}^{n} d_k^2 C_k^2\left(f\right) \log^2 k=\int_{0}^{1} f\left(x\right)P_n\left(d,c,x\right)dx. \notag
	\end{eqnarray}
	If $f\in V\left(0,1\right),$ then (see (\ref{3.2})), considering (\ref{3.1}), we get
	
	\begin{eqnarray}
		\label{3.4} \qquad &&\left|\sum_{i=1}^{n-1} \left(f\left(\frac{i}{n}\right)-f\left(\frac{i+1}{n}\right)\right) \int_{0}^{i/n}P_n\left(d,c,x\right)dx\right| \\
		&&\leq\sum_{i=1}^{n-1} \left|f\left(\frac{i}{n}\right)-f\left(\frac{i+1}{n}\right)\right| \max_{1\leq i \leq n} \left|\int_{0}^{i/n}P_n\left(d,c,x\right)dx\right|      \notag \\
		&& \leq {\mathop{V}}  \left(f\right) G_n\left(d,c\right)=O \left(T_n \left(c\right)\right).   \notag
	\end{eqnarray}
	Further, according to lemma \ref{lemma1.2}
	\begin{eqnarray}
		\label{3.5} \qquad &&\left|\sum_{i=1}^{n} \int_{\left(i-1\right)/n}^{i/n} \left(f\left(x\right)-f\left(\frac{i}{n}\right) \right) P_n\left(d,c,x\right)dx\right| \\
		&&\leq\sum_{i=1}^{n} \sup_{x\in \left[\frac{i-1}{n},\frac{i}{n}\right]} \left|f\left(x\right)-f\left(\frac{i}{n}\right)\right|	\int_{\left(i-1\right)/n}^{i/n}\left|P_n\left(d,c,x\right)\right|dx	\notag \\
		&=&O(1){\mathop{V}}\left(f\right) T_n\left(d,c\right)=O\left(1\right)  T_n\left(d,c\right).	\notag
	\end{eqnarray}
	It is easy to see that (see (\ref{1.4}))
	\begin{equation*}
	\left|f(1)\int_{0}^{1}P_n\left(d,c,x\right)\right|=O(1)G_n(d,c)=O(1)T_n(d,c).
	\end{equation*}
	Considering (\ref{3.4}) and (\ref{3.5}), from (\ref{3.2}) and (\ref{3.3}), we get
	\begin{eqnarray}
		\qquad \sum_{k=1}^{n} d_k^2 C_k^2\left(f\right) \log^2 k = O\left(1\right)T_n\left(d,c\right)=O(1)\left(\sum_{k=1}^{n}d_k^2C_k^2\left(f\right) \log^2 k\right)^{1/2}.	\notag
	\end{eqnarray}
	It follows that
	$$\sum_{k=1}^{\infty} d_k^2 C_k^2\left(f\right) \log^2 k < +\infty.$$
	Theorem \ref{theorem3.1} is proved.
	
\end{proof}

\begin{theorem}
	\label{theorem3.}
	Let $\varphi_n$ be an ONS on $[0,1]$ and   $(d_n)$ is a given sequence of numbers. If for any $\left(a_n\right)\in l_2$
	$$G_n(d,a)=O(1)T_n(d,a),$$
then $(d_n)$ is a multiplier of convergence with respect to class $V(0,1),$ or the series	
$$\sum_{n=1}^{\infty}d_n C_n(f)\varphi_n(x)$$
converges a.e. on $[0,1]$ for every $f\in V(0,1).$
	
Validity of Theorem \ref{theorem3.} follows from Theorem \ref{theorem3.1} and \ref{theorem1.1}.	
	
\end{theorem}

\begin{theorem}\label{theorem3.2}
	Let $\varphi_n$ be ONS on $[0,1]$ and  $(d_n)$ is a given bounded decreasing sequence of numbers. If for some $\left(b_n\right)\in l_2$
	$$\lim \sup_{n \rightarrow \infty}  \frac{G_n\left(b\right)}{T_n\left(b\right)}=+\infty.$$
	Then, there exist function $f_0\in A$, such that
	$$\sum_{k=1}^{\infty} d_k^2 C_k^2\left(f_0\right) \log^2 k = +\infty.$$
\end{theorem}
\begin{proof}
	In first case we suppose
	\begin{equation*}
	\lim \sup_{n \rightarrow \infty} \frac{\left|\int_{0}^{1}P_n\left(d,b,x\right)dx\right|}{T_n(b)}=+\infty.
	\end{equation*}
	If  $f_0=1,\text{ }x\in[0,1],$ then using the Cauchy inequality we get 
	\begin{eqnarray*}
	\left|\int_{0}^{1}P_n\left(d,b,x\right)dx\right| &=&  \left|\sum_{k=1}^{n}d_k^2b_k \log^2 k \int_{0}^{1} \varphi_k(x) \right|=\left|\sum_{k=1}^{n}d_k^2b_k \log^2 k C_k(f_0) \right| \\
	&\leq& \left(\sum_{k=1}^{n}d_k^2b_k^2 \log^2 k \right)^{1/2} \left(\sum_{k=1}^{n}d_k^2 C_k^2(f_0) \log^2 k \right)^{1/2} \\
	&=& T_n(b) \left(\sum_{k=1}^{n}d_k^2 C_k^2(f_0) \log^2 k \right)^{1/2}.
	\end{eqnarray*}
	Consequently 
	$$\lim_{n \rightarrow \infty}\left(\sum_{k=1}^{n}d_k^2 C_k^2(f_0) \log^2 k \right)^{1/2}=\lim_{n \rightarrow \infty}\sup\frac{\left|\int_{0}^{1}P_n\left(d,b,x\right)dx\right|}{T_n\left(b\right)}=+\infty.$$
	As $f_0\in A $  Theorem 2 holds.
	
	Next  we suppose that
	$$\left|\int_{0}^{1}P_n\left(d,b,x\right)dx\right|=O(1)T_n(b).$$
Let $1\leq i_n < n$ be an integer, such that 
	$$G_n\left(b\right)=\max_{1\leq i \leq n} \left|\int_{0}^{i/n}P_n\left(d,b,x\right)dx\right|=\left|\int_{0}^{i_n/n}P_n\left(d,c,x\right)dx\right|.$$
	
	Suppose that for some sequence $b=\left(b_k\right)\in l_2$
	\begin{eqnarray}
		\label{3.6}    \lim_{n \rightarrow \infty} \sup \frac{G_n\left(b\right)}{T_n\left(b\right)}=+\infty.
	\end{eqnarray}
Consider the sequence of functions	
	\begin{equation*}
		f_{n}\left( x\right) =\left\{ 
		\begin{array}{ccc}
			0, & \text{when} & x\in \left[ 0,\frac{i_{n}}{n}\right] \\ 
			1, & \text{when} & x\in \left[ \frac{i_{n}+1}{n},1\right] \\ 
			\text{continuous and linear, } & \text{when} & x\in \left[ \frac{i_{n}}{n}, \frac{i_{n}+1}{n} \right].
		\end{array}%
		\right.
	\end{equation*}
Let $A$ be the class of absolutely continuous functions. Then 
	$$\left\Vert f_n\right\Vert_A = \int_{0}^{1}\left|f_{n}^{'} \left(x\right)\right|dx+\left \Vert f_n\left(x\right) \right \Vert_C=2.$$
Furthermore
	\begin{eqnarray}
		\label{3.7} && \left|\sum_{i=1}^{n-1} \left(f_n\left(\frac{i}{n}\right)-f_n\left(\frac{i+1}{n}\right)\right) \int_{0}^{i/n}P_n\left(d,b,x\right)dx\right| \\
		&&=\left|\int_{0}^{{i_n}/n}P_n\left(d,b,x\right)dx\right|=G_n\left(b\right). \notag
	\end{eqnarray}
Then if $x\in \left[\frac{i-1}{n}, \frac{i}{n}\right]$
	\begin{equation*}
		\left|f_{n}\left( x\right)-f_{n}\left( \frac{i}{n}\right)\right| \left\{ 
		\begin{array}{ccc}
			\leq 1, & \text{if} & i=i_{n}+1, \\ 
			0, & \text{if} & i \neq i_{n}+1, \\ 
		\end{array}%
		\right.
	\end{equation*}
	we have (see lemma\ref{lemma1.2})
	\begin{eqnarray}
		\label{3.8}  &&\left|\sum_{i=1}^{n} \int_{\left(i-1\right)/n}^{i/n} \left(f\left(x\right)-f\left(\frac{i}{n}\right) \right) P_n\left(d,b,x\right)dx\right| \\
		&&\leq \int_{ {i_n}/n}^{(i_n+1)/n} \left|P_n\left(d,b,x\right)\right|dx = O\left(1\right)T_n\left(b\right).\notag
	\end{eqnarray}
	Consequently from equality (\ref{3.2}) when $f\left(x\right)=f_n\left(x\right)$ and $P_n\left(d,a,x\right)=P_n\left(d,b,x\right),$ considering (\ref{3.7}) and (\ref{3.8}), we get
	$$\left|\int_{0}^{1}f_n\left(x\right)P_n\left(d,b,x\right)\right|dx \geq G_n\left(b\right)-O\left(1\right)T_n\left(b\right).$$
	
	From here and from (\ref{3.6}) we have
	$$\lim \sup_{n\rightarrow \infty}\frac{\left|\int_{0}^{1}f_n\left(x\right)P_n\left(d,b,x\right)dx\right|}{T_n\left(b\right)}=+\infty.$$
	
	Since 
	$$U_n\left(f\right)=\frac{1}{T_n\left(b\right)}\int_{0}^{1}f\left(x\right)P_n\left(d,b,x\right)dx$$
	is a sequence of linear bounded functionals on $A$, then by the Banach-Steinhaus theorem, there exists a function $f_0\in A$ such that
	$$\lim \sup_{n\rightarrow \infty}\frac{\left|\int_{0}^{1}f_0\left(x\right)P_n\left(d,b,x\right)dx\right|}{T_n\left(b\right)}=+\infty.$$
	Further using the Couchy inequality
	\begin{eqnarray}
		\left|\int_{0}^{1}f_0\left(x\right)P_n\left(d,b,x\right)dx\right| &=& \left|\sum_{k=1}^{n} d_k^2 b_k \log^2 k \int_{0}^{1}f_0\left(x\right)\varphi_k\left(x\right)dx\right| \notag \\
		&=& \left|\sum_{k=1}^{n} d_k^2 b_k \log^2 k C_k\left(f_0\right)\right|\notag \\
		&\leq& \left(\sum_{k=1}^{n} d_k^2 b_k^2\log^2 k\right)^{1/2} \left(\sum_{k=1}^{n}d_k^2C_k^2\left(f_0\right)\log^2 k \right)^{1/2}	\notag \\
		&=& T_n\left(b\right) \left(\sum_{k=1}^{n}d_k^2 C_k^2\left(f_0\right)\log^2 k\right)^{1/2}.   \notag
	\end{eqnarray}
	From here 
	$$\left(\sum_{k=1}^{n}d_k^2 C_k^2\left(f_0\right)\log^2 k \right)^{1/2}\geq \frac{\left|\int_{0}^{1}f_0\left(x\right)P_n\left(d,b,x\right)dx\right|}{T_n\left(b\right)}$$
	and therefore, 
	$$\sum_{k=1}^{\infty}d_k^2 C_k^2\left(f_0\right)\log^2 k =+\infty.$$
	Theorem \ref{theorem3.2} is proved.
	
\end{proof}

Finally the following theorem holds:

\begin{theorem}\label{theorem3.3}
Let $(\varphi_n)$ be ONS on $[0,1]$, $\int_{0}^{1}\varphi_n(x)dx=0,$ $n=1,2, \dots, $ such that uniformly for $x\in [0,1]$
\begin{equation}
\label{3.9} \int_{0}^{x} \varphi_n(y)dy=O\left(\frac{1}{n}\right)
\end{equation}
and $(d_n)$ is an arbitrary non-decreasing sequence of numbers such that 
$$\lim_{n \rightarrow \infty}d_n=+\infty \text{ and } d_n=O\left(n^{\gamma}\right), \text{ } 0<\gamma<1.$$
Then for any $f\in V$  the series 
$$\sum_{n=1}^{\infty}d_n C_n(f)\varphi_n(x)$$
is convergent a.e. on $[0,1].$
\end{theorem}

\begin{proof}
	According to the condition of Theorem \ref{theorem3.3} and using the Cauchy inequality we get (see (\ref{1.4})) 
	\begin{eqnarray}
	\label{3.10} \left|\int_{0}^{x}P_n(d,c,y)dy\right|&=&\left|\sum_{k=1}^{n}d_k^2 C_k(f) \log ^2 k \int_{0}^{x}\varphi_k(y)dy\right| \\ \notag &=&O(1)\sum_{k=1}^{n}d_k^2 \left|C_k(f)\right| \log ^2 k \frac{1}{k} \\
	\notag &=& O(1)\left(\sum_{k=1}^{n}d_k^2 C_k^2(f) \log ^2 k\right)^{1/2}\left(\sum_{k=1}^{n}d_k^2 \log ^2 k \frac{1}{k^2}\right)^{1/2}   \\
	\notag &=& O(1) T_n(c) \left(\sum_{k=1}^{n}d_k^2 \log ^2 k \frac{1}{k^2}\right)^{1/2}  \\
	\notag &=& O(1) T_n(c) \left(\sum_{k=1}^{n} \frac{k^{2 \gamma}}{k^2} \right)^{1/2} = O(1) T_n(c)
	\end{eqnarray}
	Next as $f\in V$ by the  Cauchy inequality (see (\ref{3.9}))
	
	\begin{eqnarray}
	\label{3.11}  && \left|\sum_{i=1}^{n}\int_{\frac{i-1}{n}}^{\frac{i}{n}}\left(f(x)-f\left(\frac{i}{n}\right)\right)P_n(d,c,x)\right|\\
	\notag &&= O(1)\sum_{i=1}^{n}\sup_{x\in\left[\frac{i-1}{n},\frac{i}{n}\right]}\left|f(x)-f\left(\frac{i}{n}\right)\right|\left(\int_{\frac{i-1}{n}}^{\frac{i}{n}}\left|P_n(d,c,x)dx\right|\right)^{1/2} \\
	\notag &&= O(1)\frac{1}{\sqrt{n}}\left(\sum_{k=1}^{n}d_k^4 C_k^2(f) \log ^4 k\right)^{1/2} \\
	&&=O(1)\frac{d_n \log n}{\sqrt{n}}\left(\sum_{k=1}^{n}d_k^2 C_k^2(f) \log ^2 k\right)^{1/2} \notag \\
	\notag &&=O(1)\frac{n^{\gamma}\log n}{\sqrt{n}}T_n(c)=O(1)T_n(c).
	\end{eqnarray}
	Using (\ref{1.2}),(\ref{3.10}) and (\ref{3.11}) from (\ref{1.3}) we receive 
	
	\begin{equation*}
	\sum_{k=1}^{n}d_k ^2 C_k ^2(f) \log ^2 k = O(1)T_n(d,c)=O(1)\left(\sum_{k=1}^{n}d_k^2 C_k^2(f) \log ^2 k\right)^{1/2}.
	\end{equation*}
	From here we conclude 
	
	\begin{equation*}
	\sum_{k=1}^{\infty}d_k^2 C_k^2(f) \log ^2 k < +\infty.
	\end{equation*}
	Finally according to the Menshov-Rademacher Theorem the series
	\begin{equation*}
	\sum_{k=1}^{\infty}d_k C_k(f) \varphi_k(x)
	\end{equation*}
	converges a.e. on $[0,1]$ and	Theorem \ref{theorem3.3} is proved.
	
\end{proof}

	It easy to see that Theorem \ref{theorem3.3} holds for trigonometric and Walsh systems (see \cite{KashinSaakyan}, ch.4, p.117, p.150).

\begin{theorem}	\label{theorem3.4} 
	Let $(h_n)$ be an increasing sequence of numbers such that $\lim\limits_{n\rightarrow \infty} h_n=+\infty$ and $h_n=O(1)\frac{\sqrt{n}}{\log (n+1)}.$ Then from any ONS $(\varphi_n)$ one can insolate a subsequence  $(\varphi_{n_k})$ such that for an arbitraty $f\in V$
	$$\sum_{k=1}^{\infty}h_k^2 C_{n_k}^2(f)\log^2k <+\infty.$$
\end{theorem}

\begin{proof}
	Without the loss of generality we can suppose that the  ONS $(\varphi_n)$ is a complet system. Then by the   Parseval equality, for any $x \in [0,1]$
	
	$$\sum_{n=1}^{\infty}\left(\int_{0}^{x} \varphi_n(u)du\right)^2=x.$$
	Consequently (Dini Theorem) for some sequence $(n_k)$ of natural numbers
	$$\sum_{n=n_k}^{\infty}\left(\int_{0}^{x} \varphi_n(u)du\right)^2 < \frac{1}{k^4}. $$
	From here uniformly with respect to $x\in[0,1]$
	\begin{equation} \label{aliona}
		\left|\int_{0}^{x} \varphi_{n_k}(u)du\right|< \frac{1}{k^2}.
	\end{equation}
	We denote $\left((a_n)\in l_2\right)$
	$$Q_m(h,a,x)= \sum_{k=1}^{m}h_k^2 a_k \log ^2 k \varphi_{n_k}(x).$$
	 Using \eqref{aliona} and Cauchy inequality we get  (see \eqref{3.10})
	 \begin{eqnarray} 	\label{nt} 
	 	&&\max_{1\leq i \leq m} \left|\int_{0}^{i/m}Q_m(h,a,x)dx\right|\\
	 	&& = \max_{1\leq i \leq m}\left|\sum_{k=1}^{m}h_k^2 a_k \log ^2 k \int_{0}^{i/m} \varphi_{n_k}(x) dx\right| \notag \\
	 	&& = O(1) \sum_{k=1}^{m}h_k^2 \left|a_k\right| \log ^2 k  \  \frac{1}{k^2} \notag  \\
	 	&& = O(1) \left(\sum_{k=1}^{m}h_k^2 a_k^2 \log ^2 k\right)^{1/2}  \left(\sum_{k=1}^{m}h_k^2  \log ^2 k \  \frac{1}{k^4}\right)^{1/2}    \notag \\
	 	&& = O(1)\left(\sum_{k=1}^{m}h_k^2 a_k^2 \log ^2 k\right)^{1/2} \left(\sum_{k=1}^{m} \frac{k\log ^2 k}{\log ^2 (k+1)k^4}\right)^{1/2}  \notag \\
	 	&& =O(1)\left(\sum_{k=1}^{m}h_k^2 a_k^2 \log ^2 k\right)^{1/2} \notag
	 \end{eqnarray}
	Next, for any $i=1,2, \dots ,m$  (see \eqref{3.11})
	
	\begin{eqnarray} \label{nt111}
		\int_{{i-1}/m}^{i/m}\left|Q_m(h,a,x)\right|dx &\leq& \frac{1}{\sqrt{m}}\left(\int_{0}^{1}Q_m ^2 (h,a,x)dx\right)^{1/2} \\ &=&O(1)\frac{1}{\sqrt{m}}\left(\sum_{k=1}^{m}h_k^4 a_k ^2 \log ^4 k \right)^{1/2} \notag \\
		&=& O(1)\frac{h_m \log m}{\sqrt{m}} \left(\sum_{k=1}^{m}h_k^2 a_k ^2 \log ^2 k \right)^{1/2} \notag \\
		&=& O(1) \frac{\sqrt{m} \log m}{\log (m+1) \sqrt{m} } \left(\sum_{k=1}^{m}h_k^2 a_k ^2 \log ^2 k \right)^{1/2}	\notag \\
		&=& O(1)  \left(\sum_{k=1}^{m}h_k^2 a_k ^2 \log ^2 k \right)^{1/2}.	\notag 
	\end{eqnarray}
	Also (see \eqref{aliona} and \eqref{nt})
	\begin{equation}
		\label{al22}
		\left|\int_{0}^{1}Q_m(h,a,x)dx\right|=O(1)  \left(\sum_{k=1}^{m}h_k^2 a_k ^2 \log ^2 k \right)^{1/2}.
	\end{equation}
	
	As it was shown in \eqref{1.3}
	\begin{eqnarray}
		\label{al09} \sum_{k=1}^{m}h_k^2 C_{n_k} ^2(f) \log ^2 k &=&\int_{0}^{1}f(x) \sum_{k=1}^{m}h_k^2 C_{n_k}(f) \log ^2 k \varphi_{n_k}(x)dx \\
		&=& \int_{0}^{1}f(x)Q_m(h,c,x)dx. \notag
	\end{eqnarray}

 Taking into account \eqref{3.2} and \eqref{al09} where $Q_m (h,c,x)=P_n (d,c,x),$ $f\in V(0,1)$ and estimates \eqref{nt}, \eqref{nt111}, \eqref{al22}  where  $a=c,$   $a_k=C_{n_k}(f),$ we obtain 
 \begin{eqnarray*}
 	\left|\sum_{k=1}^{m}h_k^2 C_{n_k}^2(f) \log ^2 k\right|&=&\left|\int_{0}^{1}f(x)Q_m(h,c,x)dx\right|\\
 	&=&O(1) \left|V(f)+f(1)\right|\left(\sum_{k=1}^{m}h_k^2 C_{n_k}^2(f) \log ^2 k\right)^{1/2}.
 \end{eqnarray*}
From here
$$\sum_{k=1}^{\infty}h_k^2 C_{n_k}^2(f) \log ^2 k < +\infty.$$

   Theorem \ref{theorem3.4} is completely proved.
   
\end{proof}

\begin{theorem}
	\label{theorem3.5}
	Let $(h_n)$ be an increasing sequence of numbers such that $$\lim\limits_{n\rightarrow \infty}h_n=+\infty \text{ and } h_n = O(1)\frac{\sqrt{n}}{\log (n+1)}.$$
	Then from any ONS $(\varphi_n)$ one can insolate a subsequence $(\varphi_{n_k}(x))$ such that for an arbitrary $f\in V$ the series
	$$\sum_{k=1}^{\infty}h_k C_{n_k}(f) \varphi_{n_k} $$ 
	is convergent a.e. on $[0,1].$
	
	The validity of Theorem \ref{theorem3.5} derives from Theorems \ref{theorem3.4} and \ref{theorem3.1}. 
	
\end{theorem}

\end{document}